\documentclass[11pt]{article}
\usepackage{amsfonts}
\usepackage{mathrsfs}

\usepackage[latin1]{inputenc}
\usepackage{amsmath,amssymb}
\usepackage{latexsym}
\usepackage{epstopdf}
\usepackage{geometry}   
\usepackage[active]{srcltx}
\usepackage{graphicx}
\usepackage{epstopdf}
\usepackage[
bookmarks=true,         
bookmarksnumbered=true, 
colorlinks=true, pdfstartview=FitV, linkcolor=blue, citecolor=blue,
urlcolor=blue]{hyperref}

\newtheorem{theorem}{Theorem}[section]

\newtheorem{lemma}[theorem]{Lemma}
\newtheorem{proposition}[theorem]{Proposition}

\numberwithin{equation}{section}
\def\R{{\mathbb R}}
\def\E{{{\mathbb E}\,}}

\def\N{{\mathbb N}}

\def\Var{{\mathop {{\rm Var\, }}}}

\def\square{{\vcenter{\vbox{\hrule height.3pt
        \hbox{\vrule width.3pt height5pt \kern5pt
           \vrule width.3pt}
        \hrule height.3pt}}}}

\def\tlint{{- \kern-0.85em \int \kern-0.2em}}
\def\dlint{{- \kern-1.05em \int \kern-0.4em}}

\def\si{\sigma}

\def\cF{{\cal F}}
\def\Om{{\Omega}}

\def\cF{{\cal F}}

\def\si{{\sigma}}

\def\Om{{\Omega}}

\def\si{{\sigma}}

\def \eref#1{\hbox{(\ref{#1})}}

\def \eref#1{\hbox{(\ref{#1})}}

\def\si{{\sigma}}

\newenvironment{proof}[1][Proof]{\noindent\textit{#1.} }{\hfill \rule{0.5em}{0.5em}}

\begin{document}

\title{Central limit theorem for functionals of two independent fractional Brownian motions}
\date{\empty}

\author{ David Nualart\thanks{ D. Nualart is supported by the NSF grant
DMS1208625.} \ \
and \ Fangjun Xu\thanks{F. Xu is supported in part by the Robert Adams Fund.}\\
Department of Mathematics \\
University of Kansas \\
Lawrence, Kansas, 66045 USA}

\maketitle

\begin{abstract}
\noindent  We prove a central limit theorem for functionals of two independent $d$-dimensional fractional Brownian motions with the same Hurst index $H$ in $(\frac{2}{d+1},\frac{2}{d})$ using the method of moments.


\end{abstract}

\section{Introduction}

Let $\left\{ B^H_t= (B^1_t,\dots,
B^d_t), t\geq 0\right\} $ be a $d$-dimensional fractional Brownian
motion (fBm) with Hurst index $H$ in  $(0,1)$. Let $B^{H,1}$ and $B^{H,2}$ be two independent copies of  $B^H$. If $Hd<2$,  then the intersection local time of $B^{H,1}$ and $B^{H,2}$ exists (see \cite{nol}) and can be defined as 
\[
\alpha(t_1, t_2)=\int^{t_1}_0\int^{t_2}_0 \delta(B^{H,1}_u-B^{H,2}_v)\, du\, dv,
\]
where $\delta$ is the Dirac delta function. For any $t_1$ and $t_2$ in $\R^+$, define
\[
X(t_1,t_2)=B^{H,1}_{t_1}-B^{H,2}_{t_2}.
\]
We see that $X=\{X(t_1,t_2), t_1,t_2\in\R^+\}$ is a $(2,d)$-Gaussian random field and satisfies the following scaling property: for any $c>0$, 
\begin{equation}
\big\{X(ct_1,ct_2), t_1,t_2\in\R^+\big\}\overset{\mathcal{L}}{=}\big\{c^HX(t_1,t_2), t_1,t_2\in\R^+\big\}.   \label{scale}
\end{equation}
If $Hd<2$, then, for any $x$  in $\R^d$ and rectangle $E=[a_1,b_1]\times[a_2,b_2]$ in $\R_+^2$, the local time $L(x, E)$ of $X$ exists and is continuous in $x$, see \cite{wu_xiao}. When $E=[0,t_1]\times[0,t_2]$, $\alpha(t_1, t_2)=L(0,E)$. Throughout this paper, we assume $Hd<2$ to ensure the existence and the continuity of $L(x, E)$. 

For any integrable function $f:\R^d\to\R$, one can easily show the following convergence in law in the space $C([0,\infty)^2)$, as $n$ tends to infinity,
\begin{equation*}
\Big\{ n^{Hd-2}\int^{nt_1}_0\int^{nt_2}_0 f(B^{H,1}_u-B^{H,2}_v)\, du\, dv,\; t_1, t_2>0 \Big\} \overset{\mathcal{L}}{\longrightarrow } \Big\{ \alpha(t_1,t_2) \int_{\R^d} f(x)\, dx,\;  t_1, t_2>0 \Big\}.
\end{equation*}
In fact, letting $E=[0,t_1]\times[0,t_2]$, using the scaling property of the process $X(u,v)$ in \eref{scale} and then applying the continuity of $L(x,E)$, we get
\begin{eqnarray*}
n^{Hd-2}\int^{nt_1}_0\int^{nt_2}_0 f(B^{H,1}_u-B^{H,2}_v)\, du\, dv
&\overset{\mathcal{L}}{=} &  n^{Hd}\int^{t_1}_0\int^{t_2}_0 f\big(n^H (B^{H,1}_u-B^{H,2}_v) \big)\, du\, dv \\
&=                                         &  n^{Hd} \int_{\R^d}f(n^H x)L(x,E)\, dx\\
&=                                         &   \int_{\R^d} f(x)L(\frac{x}{n^H},E)\, dx\\
&\overset{a.s.}{\longrightarrow } &\alpha(t_1,t_2) \int_{\R^d} f(x)\, dx,
\end{eqnarray*}
where $\overset{a.s.}{\longrightarrow }$ denotes the almost sure convergence.

If we assume $\int_{\R^d} f(x)\, dx=0$, then the random variable
\[
n^{Hd-2} \int^{nt_1}_0\int^{nt_2}_0 f(B^{H,1}_u-B^{H,2}_v)\, du\, dv
\]
converges in law to $0$ as $n$ tends to infinity. It is natural to ask if there is a $\beta>Hd-2$ such that 
\[
n^{\beta}\int^{nt_1}_0\int^{nt_2}_0 f(B^{H,1}_u-B^{H,2}_v)\, du\, dv
\]
converges to a nontrivial random variable. This will be proved to be true.
In order to  formulate this result  we introduce the
following space of functions. Fix a number $\beta\in(0,2)$, define
\[
H^{\beta}_0=\Big\{f\in L^1(\R^d):\, \int_{\R^d} |f(x)||x|^{\beta}\, dx<\infty \quad\text{and}\quad \int_{\R^d} f(x)\, dx=0 \Big\}. \label{beta}
\]

For any $f\in H^{\beta}_0$, by Lemma 4.1 in \cite{hnx}, the quantity
\[
\|f\|^2_{\beta}=-\int_{\R^{2d}} f(x)f(y) |y-x|^{-\beta}\, dx\, dy
\]
is finite and nonnegative. The next theorem is the main result of this paper.
\begin{theorem}  \label{thm1} Suppose $\frac{2}{d+1}<H<\frac{2}{d}$ and $f\in H^{\frac{2}{H}-d}_0$. Then, for any $t_1$ and $t_2>0$,
\[
n^{\frac{Hd-2}{2}} \int^{nt_1}_0\int^{nt_2}_0 f(B^{H,1}_u-B^{H,2}_v)\, du\, dv\overset{\mathcal{L}}{\longrightarrow} \sqrt{D_{H,d}}\, \|f\|_{\frac{2}{H}-d}\, \sqrt{\alpha(t_1, t_2)}\, \zeta,
\]
as $n\to\infty$, where 
\[
D_{H,d}=\frac{4}{(2\pi)^{\frac{d}{2}}}\int_{0}^\infty\int_{0}^\infty
(u^{2H}+v^{2H})^{-\frac{d}{2}}  \big( 1- e^{ -\frac{1}{2}\frac{1} {u^{2H}+v^{2H}} }\big) du\, dv
\]
and $\zeta$ is a standard normal random variable independent of the processes $B^{H,1}$ and $B^{H,2}$.
\end{theorem}

In \cite{hnx}, Hu, Nualart and I proved the following functional central limit theorem
\[
  \Big\{  n^{\frac{Hd-1}{2}}\int_{0}^{nt} f(B^H(s))\, ds\,,
  \ t\ge 0\Big\} \ \ \overset{\mathcal{L}}{\longrightarrow } \ \
  \Big\{  \sqrt{C_{H,d}} \, \| f\|_{\frac{1}{H}-d}\, W (L_{t}(0))\,, t\ge
  0\Big\},
\]
where $W$ is a real-valued standard Brownian motion independent of $B^H$ and $L_t(x)$ is the local time of $B^H$. This paper can be viewed as an extension of the result in \cite{hnx}. To prove our main result Theorem \ref{thm1}, we use the method of moments. Some techniques in \cite{hnx} will be used, but new ideas are needed. The  basic idea of the approach used in this paper is to  apply the method of moments to a functional. When dealing with an integral on $[0,t_1]^{2m}\times[0,t_2]^{2m}$, with respect to the measure $du_1  \cdots du_{2m}dv_1  \cdots dv_{2m}$ , we make the change of variables 
$w_{2k-1}=n(u_{2k}-u_{2k-1})$, $w_{2k}=u_{2k}$, $s_{2k-1}=n(v_{2k}-v_{2k-1})$, $s_{2k}=v_{2k}$, $1\le k\le m$. Then, the increments of $B^{H,1}-B^{H,2}$ in small rectangles will be responsible for the independent noise appearing in the limit. This methodology could be applied to other examples of functionals and multi-parameter processes. 

Note that the constant $D_{H,d}$ is finite for any $H>\frac{2}{d+2}$.  We conjecture that our result is also true for $\frac{2}{d+2}<H<\frac{2}{d}$, but we have not been able to show our result in the case $H\leq\frac{2}{d+2}$. The main reason is that we need to use Fourier analysis in the proof of our result. For example, we need to assume $Hd>2-H$ in Lemma \ref{lema3}.

In the Brownian motion case ($H=\frac{1}{2}$ and $d=3$), the functional version of Theorem \ref{thm1} can be proved using a theorem by Weinryb and Yor \cite{weinryb_yor}. A second order result for two independent Brownian motions in the critical case $d=4$ and $H=\frac{1}{2}$ was proved by Le Gall \cite{LeGall}. However, not nearly as much has been done for the case $H\neq \frac{1}{2}$ and $Hd=2$.  The general asymptotic results for additive functionals of $k$ independent Brownian motions were obtained by Biane \cite{biane}. This paper extends some results in \cite{biane} to fractional Brownian motions. General extensions are still largely unknown. 

After some preliminaries in Section 2, Section 3 is devoted to the proof of Theorem \ref{thm1},
based on the method of moments. Throughout this paper, if not mentioned otherwise, the letter $c$, 
with or without a subscript, denotes a generic positive finite
constant whose exact value is independent of $n$ and may change from
line to line. We use $\iota$ to denote $\sqrt{-1}$.

\section{Preliminaries}
Let $\left\{ B^H_t= (B^1_t,\dots,
B^d_t), t\geq 0\right\} $ be a $d$-dimensional fractional Brownian
motion with Hurst index $H$ in $(0,1)$, defined on some probability space $(\Om, \cF, P)$. That is, the components
of $B^H$ are independent centered Gaussian processes with covariance function
\[
\E\big(B^i_t B^i_s\big)=\frac{1}{2}\big(t^{2H}+s^{2H}-|t-s|^{2H}\big).
\]  
We shall use the following property of the fractional Brownian motion $B^H$.
\begin{lemma} \label{lnd}
Given $n\ge 1$, there exist two constants $c_{1}$ and $c_{2}$ depending only on  $n$, $H$ and $d$, such that for any $0=s_0<s_1 <\cdots < s_n $ and $x_i \in \mathbb{R}^d$, $1\le i \le n$, we have
\[
c_{1} \sum_{i=1}^n |x_i|^2 (s_i -s_{i-1})^{2H} \le \mathrm{Var} \Big( \sum_{i=1}^n x_i  \cdot (B^H_{s_i} -B^H_{s_{i-1}}) \Big) \le c_{2}  \sum_{i=1}^n |x_i|^2 (s_i -s_{i-1})^{2H}.
\]
\end{lemma} 
\begin{proof} The second inequality is obvious. So it suffices to show the first one, which follows from the local nondeterminism property of the fractional Brownian motion; see, e.g., \cite{biane} and \cite{hnx}. 
\end{proof}

The inequalities in Lemma \ref{lnd} can be rewritten as
\begin{equation} \label{eq1}
c_{1} \sum_{i=1}^n \Big|\sum_{j=i}^n x_j \Big|^2  (s_i -s_{i-1})^{2H} \le \mathrm{Var} \Big( \sum_{i=1}^n x_i  \cdot B^H_{s_i}  \Big) \le c_{2}  \sum_{i=1}^n \Big|  \sum_{j=i}^n x_j \Big|^2(s_i -s_{i-1})^{2H}.
\end{equation}
The next lemma gives a formula for the moments of the random variable $\sqrt{\alpha(t_1, t_2)}\, \zeta$ appearing in Theorem \ref{thm1}. 
\begin{lemma} \label{lema2.1} For any $p\in \N$, 
\begin{eqnarray}
\E \big[\sqrt{\alpha(t_1, t_2)}\, \zeta \big]^{p} \label{2.1}
&=&\begin{cases} 
\frac {(2m-1)!! }{(2\pi )^{\frac{md}{2}}} \int_{E^m} \big(\det A (u, v)\big)^{-\frac{1}{2}}  du\, dv
&\hbox{if $p=2m$},\\
\\
0 & \hbox{otherwise}, 
\end{cases}   \notag
\end{eqnarray}
where $E=[0,t_1]\times[0,t_2]$ and $A(u,v)$ is the covariance matrix of the Gaussian random field
\[
\big(B^{H,1}_{u_i}-B^{H,2}_{v_i}, 1\leq i\leq m\big).
\]
\end{lemma}

\begin{proof} This follows easily from the properties the normal distribution and the intersection local time $\alpha(t_1,t_2)$.
\end{proof}

\medskip

\section{Proof of Theorem \ref{thm1}}

By the scaling property of $X(t_1,t_2)$ in \eref{scale}, we see that, as random variables,
\[
n^{\frac{Hd-2}{2}} \int^{nt_1}_0\int^{nt_2}_0 f(B^{H,1}_u-B^{H,2}_v)\, du\, dv\overset{\mathcal{L}}{=}n^{\frac{2+Hd}{2}}\int^{t_1}_0\int^{t_2}_0 f\big(n^H(B^{H,1}_u-B^{H,2}_v)\big)\, du\, dv.
\]
Therefore, it suffices to show Theorem \ref{thm1} for the random variable
\[
F_{n}(t_1, t_2)=n^{\frac{2+Hd}{2}}\int^{t_1}_0\int^{t_2}_0 f\big(n^H(B^{H,1}_u-B^{H,2}_v)\big)\, du\, dv.
\]
The proof of Theorem \ref{thm1} will be done in two steps. We first show tightness and then establish the convergence of moments. 

\subsection{Tightness}

Tightness will be deduced from the following result.
\begin{proposition} \label{tight}
For any integer $m\geq 1$, there exists a positive constant $C$ independent of $n$ such that 
\[
\E\big[F_{n}(t_1,t_2)\big]^{2m}\leq C\,
\Big[\int_{\R^{2d}}|f(x)f(y)| |y|^{\frac{2}{H}-d} \,
dx\, dy\Big]^{m}.
\]
\end{proposition}

\begin{proof}  Note that
\begin{equation}
\E\big[F_{n}(t_1,t_2)\big]^{2m} =  n^{m(2+Hd)} \E \Big[\int_{E^{2m}}\prod_{i=1}^{2m} f\big(n^H (B^{H,1}_{u_i}-B^{H,2}_{v_i})\big)\, du\, dv\Big], \label{t1.e.1}
\end{equation}
where $E=[0,t_1]\times[0,t_2]$.  

Using Fourier analysis and making proper change of variables, 
\begin{align}
&  (2\pi n^H)^{2md}\E \Big[\int_{E^{2m}}\prod_{i=1}^{2m} f\big(n^{H} (B^{H,1}_{u_i}-B^{H,2}_{v_i})\big)\, du\, dv \Big] \notag \\
 =  &\, \int_{\R^{4md}} \int_{E^{2m}}\prod_{i=1}^{2m} f(z_i)
                \exp\Big\{-\frac{1}{2}\Var \Big(\sum^{2m}_{i=1}\xi_i \cdot \big(B^{H,1}_{u_i}-B^{H,2}_{v_i}\big)\Big) - \iota\sum^{2m}_{i=1}\frac{z_i\cdot \xi_i}{n^H}\Big\}\, du\, dv\, d\xi\, dz\notag\\
 =  &\,   \int_{\R^{4md}} \int_{E^{2m}}\prod_{i=1}^{2m} f(z_i)  \prod^{2m}_{i=1} \Big(e^{-\iota \frac{z_i\cdot \xi_i}{n^H}}-1\Big)  \notag \\
             &  \qquad  \times   \exp\Big\{-\frac{1}{2}\Var\Big(\sum^{2m}_{i=1}\xi_i \cdot B^{H}_{u_i}\Big)-\frac{1}{2}\Var\Big(\sum^{2m}_{i=1}\xi_i \cdot B^{H}_{v_i}\Big)\Big\} \, du\, dv\, d\xi\, dz,   \label{t1.e.2}
      \end{align}
where  in the last equality we used the fact that $\int_{\R^d}
f(x)\, dx=0$.

Let $t=\max\{t_1,t_2\}$ and $\mathscr{P}$ be the set consisting of all permutations of $\{1,2,\dots, 2m\}$.  Set
 \[
I_{t}(\xi)= \int_{ [0,t]^{2m}} \exp\Big\{ -\frac 12 \mathrm{Var} \big( \sum_{j=1}^{2m} \xi_j \cdot B^H_{u_j}  \big)\Big\}\, du.
 \]
For any $\sigma\in\mathscr{P}$, define
\[
I_{t}^\sigma (\xi)= \int_{ D_\sigma } \exp \Big\{-\frac 12 \mathrm{Var} \big( \sum_{j=1}^{2m} \xi_j \cdot B^H_{u_j}  \big)\Big\}\, du,
\]
where $D_\sigma=\{ u\in [0,t] ^{2m}: u_{\sigma(1)} <\dots < u_{\sigma(2m)} \}$. 

Therefore, $I_{t}(\xi)$ can be decomposed as
\begin{equation}
I_{t}(\xi)= \sum_{\sigma\in \mathscr{P}}  I_{t}^{\sigma}(\xi).   \label{t1.e.3}
\end{equation}
For simplicity of notation, set
\begin{equation} \label{Phin}
\Phi_n(\xi, z)= n^{m(2-Hd)} \prod_{i=1}^{2m} \big|f(z_i) \big|
\prod_{i=1}^{2m}  \big| e^{i \frac {z_i\cdot \xi_i}{n^H}} -1\big|.
\end{equation}

From \eref{t1.e.1}, \eref{t1.e.2} and \eref{t1.e.3},  we can write
\begin{align}
\E\big[F_{n}(t_1,t_2)\big]^{2m}
&\leq  c_1 \int_{\R^{4md}} \Phi_n(\xi, z) \big(I_{t}(\xi)\big)^2 d\xi\, dz \leq  c_2 \sum_{\sigma\in\mathscr{P}} \int_{\mathbb{R}^{4md}}  \Phi_n(\xi,z) \big(I_{t}^\sigma(\xi)\big)^2\, d\xi\, dz. \label{t1.e.4}
\end{align}
Observe that
\begin{align*}
\big(I_{t}^\sigma(\xi) \big)^2
&= \int_{ D_\sigma\times D_{\sigma}} \exp \bigg\{  -\frac 12 \mathrm{Var} \Big( \sum_{j=1}^{2m} \xi_j \cdot B^H_{u_j} \Big)-\frac 12 \mathrm{Var} \Big( \sum_{j=1}^{2m} \xi_j \cdot B^H_{v_j}  \Big)\bigg\}\, du\, dv \notag \\
&= \int_{ \widehat{D}} \exp \bigg\{ -\frac 12 \mathrm{Var} \Big( \sum_{j=1}^{2m} \xi_j \cdot B^H_{w_j} \Big)-\frac 12 \mathrm{Var} \Big( \sum_{j=1}^{2m} \xi_j \cdot B^H_{s_j}  \Big)\bigg\}\, dw\, ds,
\end{align*}
where $\widehat{D}=\big\{w,s\in [0,t] ^{2m}: w_1 <\dots < w_{2m}\; \text{and}\; s_1 <\dots <s_{2m}\big\}$.

Using the second inequality in \eref{eq1}, we obtain that 
\[
\int_{\mathbb{R}^{4md}}  \Phi_n(\xi,z) \big(I_{t}^\sigma(\xi) \big)^2\, d\xi\, dz
\] 
is less than or equal to
\[
\int_{\mathbb{R}^{4md}} \int_{ \widehat{D}} \Phi_n(\xi,z) \exp\Big\{-\frac{\kappa_H}{2}\sum^{2m}_{i=1} |\sum^{2m}_{j=i}\xi_j|^2\big[(w_i-w_{i-1})^{2H}+(s_i-s_{i-1})^{2H}\big]\Big\}\, dw\, ds\, d\xi \, dz,
\]
with the convention $w_0=s_0=0$.

Making the change of variables $\eta_i=\sum\limits^{2m}_{j=i}\xi_j$, $u_i=w_i-w_{i-1}$ and $v_i=s_i-s_{i-1}$ for $i= 1,\dots, 2m$ gives
\begin{align}
&\int_{\mathbb{R}^{4md}}  \Phi_n(\xi,z) (I_{t}^\sigma(\xi) )^2\, d\xi\, dz \notag\\
 &\quad \leq n^{m(2-Hd)} \int_{\R^{4md}} \int_{[0,t]^{4m}}\big|\prod_{i=1}^{2m} f(z_i)\big| \prod^{2m}_{i=1}\big|\exp\big(\iota \frac{z_i}{n^H}\cdot (\eta_{i+1}-\eta_i)\big)-1\big| \notag\\
         &\quad\quad\quad \times  \exp\Big\{-\frac{\kappa_H}{2}\sum^{2m}_{i=1} |\eta_i|^2(u_i^{2H}+v_i^{2H})\Big\}\, d\eta\, du\, dv\, dz, \label{t.e.4}
\end{align}
with the convention $\eta_{2m+1}=0$.

Let $\sqrt{\kappa_H}X_1, \dots, \sqrt{\kappa_H}X_{2m}$ be independent
copies of the $d$-dimensional standard normal random vector and $X_{2m+1}=0$. Then inequality \eref{t.e.4} can be written as 
\begin{align} \label{t.e.4.1}
&\int_{\mathbb{R}^{4md}}  \Phi_n(\xi,z) (I_{t}^\sigma(\xi) )^2\, d\xi\, dz \notag\\
& \leq    c_3\, n^{m(2-Hd)} \E\bigg[\int_{\R^{2md}} \int_{[0,t]^{4m}}\prod_{i=1}^{2m} \big|f(z_i)\big| \prod_{i=1}^{2m}(u^{2H}_i+v^{2H}_{i})^{-\frac{d}{2}}  \notag \\
          &\qquad \times \prod^{2m}_{i=1}\Big|\exp\Big(\iota \frac{z_i}{n^H}\cdot \big(\frac{X_{i+1}}{\sqrt{u^{2H}_{i+1}+v^{2H}_{i+1}}}-\frac{X_i}{\sqrt{u^{2H}_i+v^{2H}_{i}}}\big)\Big)-1\Big|  \, du\, dv\, dz\bigg]. 
\end{align}   

To make use of the independence of $X_1, X_2,\dots, X_{2m}$, we replace the terms
\[ 
\Big|\exp\Big(\iota \frac{z_i}{n^H}\cdot \big(\frac{X_{i+1}}{\sqrt{u^{2H}_{i+1}+v^{2H}_{i+1}}}-\frac{X_i}{\sqrt{u^{2H}_i+v^{2H}_{i}}}\big)\Big)-1\Big|, \quad i=2,4,\dots,2m
\]
on the right hand side of inequality \eref{t.e.4.1} with $2$ and then obtain
\begin{align}
&\int_{\mathbb{R}^{4md}}  \Phi_n(\xi,z) (I_{t}^\sigma(\xi) )^2\, d\xi\, dz \notag\\
& \leq    c_4\, n^{m(2-Hd)} \E\bigg[\int_{\R^{2md}} \int_{[0,t]^{4m}}\prod_{i=1}^{2m} \big|f(z_i)\big| \prod_{i=1}^{2m}(u^{2H}_i+v^{2H}_{i})^{-\frac{d}{2}}  \notag \\
          &\qquad \times \prod^{m}_{k=1}\Big|\exp\Big(\iota \frac{z_i}{n^H}\cdot \big(\frac{X_{2k}}{\sqrt{u^{2H}_{2k}+v^{2H}_{2k}}}-\frac{X_{2k-1}}{\sqrt{u^{2H}_{2k-1}+v^{2H}_{2k-1}}}\big)\Big)-1\Big|  \, du\, dv\, dz\bigg]  \notag\\        
&     =   c_4 \bigg(n^{2-Hd}\int_{\R^{2d}} \int_{[0,t]^4}\prod_{i=1}^{2} \big|f(z_i)\big| \prod_{i=1}^{2}\big(\sqrt{u^{2H}_i+v^{2H}_{i}}\big)^{-d}  \notag\\
          &\qquad \times \E\Big[\Big|\exp\Big(\iota \frac{z_1}{n^H}\cdot \big(\frac{X_2}{\sqrt{u^{2H}_2+v^{2H}_2}}-\frac{X_1}{\sqrt{u^{2H}_1+v^{2H}_1}}\big)\Big)-1\big|\Big]  \, du_1\, du_2\, dv_1\,
          dv_2\, dz_1\, dz_2\bigg)^{m}  \notag\\
& \leq    c_5\,  t^{2-Hd} \Big(\int_{\R^{2d}}|f(z_1)f(z_2)|z_1|^{\frac{2}{H}-d}\, dz_1\, dz_2\Big)^m, \label{t.e.5}
\end{align}
where in the last inequality we used Lemmas \ref{lema2} and \ref{lema3}.

Combining \eref{t.e.5} and \eref{t1.e.4} gives the desired inequality. \end{proof}

\medskip

\subsection{Convergence of odd moments}

Let $t=\max\{t_1,t_2\}$. For any $p\in\N$, $y\in\R^d$ and $\xi=(\xi_1,\dots,\xi_p)\in(\R^d)^p$, define
\begin{equation} \label{Phi}
\Phi_{n,p}(\xi,y)=n^{\frac{p(2-Hd)}{2}}\prod^{p}_{j=1} \big|f(y_j)   (e^{ -\iota \frac{y_j\cdot \xi_j}{n^H}}-1)\big|
\end{equation}
and
\[
H_{n,p}=\int_{\R^{2pd}}\int_{[0,t]^{2p}} \Phi_{n,p}(\xi,y)\exp\bigg\{ -\frac{1}{2}\Var\Big( \sum^{p}_{j=1} \xi_j\cdot  \big(B^{H,1}_{u_j} - B^{H,2}_{v_j} \big)\Big)\bigg\}\,  du\, dv\, d\xi\, dy.
\]
Note that for $p=2m$, $\Phi_{n,p}(\xi,y)$ is precisely the function defined \eref{Phin}. Also in the proof of Proposition \ref{tight}, we have that $\E\big[F_{n}(t_1,t_2)\big]^{2m}\leq c\, H_{n,2m}$. We are going to use see that if $p$ is odd, then $H_{n,p}$ converges to zero as $n$ tends to infinity. This will imply the convergence of odd moments.

\begin{proposition} \label{odd} If $p$ is odd, then
\[
\lim\limits_{n\to\infty}H_{n,p}=0.
\]
\end{proposition}
\begin{proof} Using similar notation as in the proof of Proposition \ref{tight}, we have
\begin{align*}
H_{n,p}
&\leq c_1\,  n^{\frac{p(2-Hd)} 2}\, \E\bigg[\int_{\R^{pd}} \int_{[0,t]^{2p}}\prod_{i=1}^{p} \big|f(y_i)\big| \prod_{i=1}^{p}(u^{2H}_i+v^{2H}_{i})^{-\frac{d}{2}}  \notag \\
&\qquad \times \prod^{p}_{i=1}\Big|\exp\Big(\iota \frac{y_i}{n^H}\cdot \big(\frac{X_{i+1}}{\sqrt{u^{2H}_{i+1}+v^{2H}_{i+1}}}-\frac{X_i}{\sqrt{u^{2H}_i+v^{2H}_{i}}}\big)\Big)-1\Big|  \, du\, dv\, dy\bigg],
\end{align*}
with the convention $X_{p+1}=0$. Since $X_1, X_2,\cdots, X_p$ are i.i.d,
\begin{align*} 
H_{n,p}
&\leq c_2\, n^{\frac{p(2-Hd)} 2}\, \E\bigg[\int_{\R^{pd}} \int_{[0,t]^{2p}}\prod_{i=1}^{p} \big|f(y_i)\big| \prod_{i=1}^{p}(u^{2H}_i+v^{2H}_{i})^{-\frac{d}{2}}  \notag \\
          &\qquad \times \prod^{\frac{p+1}{2}}_{k=1}\Big|\exp\Big(\iota \frac{y_{2k-1}}{n^H}\cdot \big(\frac{X_{2k}}{\sqrt{u^{2H}_{2k}+v^{2H}_{2k}}}-\frac{X_{2k-1}}{\sqrt{u^{2H}_{2k-1}+v^{2H}_{2k-1}}}\big)\Big)-1\Big|  \, ds\, dt\, dy\bigg]\\            
&\leq c_3\, n^{\frac{(2-Hd)} 2} \Big(\int_{\R^{2d}}|f(y_1)f(y_2)|y_1|^{\frac{2}{H}-d}\, dy_1\, dy_2\Big)^\frac{p-1}{2} \notag \\
          &\qquad \times  \E\bigg[\int_{\R^{d}} \int_{[0,t]^{2}} |f(y_p)| (u^{2H}_p+v^{2H}_p)^{-\frac{d}{2}}\Big|\exp\big(-\iota \frac{y_p}{n^H}\cdot \frac{X_{p}}{\sqrt{u^{2H}_{p}+v^{2H}_{p}}}\big)-1\Big|  \, du_p\, dv_p\, dy_p\bigg]\\ 
&\leq c_4\, n^{\frac{(Hd-2)} 2} \Big(\int_{\R^{2d}}|f(y_1)f(y_2)|y_1|^{\frac{2}{H}-d}\, dy_1\, dy_2\Big)^\frac{p-1}{2} \Big(\int_{\R^{d}}|f(y_p)||y_p|^{\frac{2}{H}-d}\, dy_p\Big),                
\end{align*}
where in the last two inequalities we used Lemmas \ref{lema2} and \ref{lema3}.   Therefore, $\lim\limits_{n\to\infty}H_{n,p}=0$.
\end{proof}

Propositions \ref{tight} and \ref{odd} show that $H_{n,p}$ is uniformly bounded in $n$. Moreover, Proposition \ref{odd} implies the following convergence of odd moments.

\begin{proposition}\label{odd1} Suppose $p$ is odd, then 
\[
\lim\limits_{n\to\infty}\E[F_n(t_1,t_2)]^{p}=0.
\]
\end{proposition}
\begin{proof}  Since $|\E(F_n(t_1,t_2))^p|\leq H_{n,p}$ for all $p$, this follows immediately from Proposition \ref{odd}.
\end{proof}

\subsection{Some technical lemmas}

To prove the convergence of even moments, we need some technical lemmas. 

Recall that 
\[
 \E\big[F_{n}(t_1,t_2)\big]^{2m} =n^{m(2+Hd)} \E \Big[
           \int_{E^{2m}} \prod_{j=1}^{2m}f\big(n^{H}B^{H,1}_{u_j}-n^{H}B^{H,2}_{v_j}\big)\, du\, dv \Big],
 \]
 where $E=[0,t_1]\times[0,t_2]$.  
 
Let $\mathscr{P}$ be the set consisting of all permutations of $I=\{1,2,\dots, 2m\}$ and
\begin{equation} \label{domain}
D=\Big\{0=u_0<u_{1}<u_{2}<\cdots<u_{2m}<t_1,\, 0=v_0<v_{1}<v_{2}<\cdots<v_{2m}<t_2\Big\}.
\end{equation}
Then
\begin{equation} \label{moment}
 \E\big[F_{n}(t_1,t_2)\big]^{2m} =(2m)!\, n^{m(2+Hd)}  \sum_{\sigma\in\mathscr{P}}\E \Big[
           \int_{D} \prod_{j=1}^{2m}f\big(n^{H}B^{H,1}_{u_j}-n^{H}B^{H,2}_{v_{\sigma(j)}}\big)\, du\, dv \Big]. 
\end{equation}

For any $\epsilon>0$, define  
\[
H^{\sigma}_{n,2m,\epsilon}=\int_{\R^{4md}}\int_{R_{\epsilon}} \Phi_{n,2m}(\xi,y)\exp\Big\{ -\frac{1}{2}\Var\Big( \sum^{2m}_{j=1} \xi_j\cdot  \big(B^{H,1}_{u_j} - B^{H,2}_{v_{\sigma(j)}} \big)\Big)\Big\}\,  du\, dv\, d\xi\, dy,
\]
where
\begin{align*}
R_{\epsilon}&=
\Big\{0<u_{1}<u_{2}<\cdots<u_{2m}<t,\, 0<v_{1}<v_{2}<\cdots<v_{2m}<t\Big\}\\
&\qquad \bigcap \Big\{|u_{2j}-u_{j-2}|<\epsilon\; \text{or}\; |v_{2j}-v_{2j-2}|<\epsilon,\; \text{for some}\; j\in \{1,2,\dots, m\} \Big\}
\end{align*}
with the convention $u_0=v_0=0$ and $t=\max\{t_1,t_2\}$.

\begin{lemma}\label{epsilon}  For any $\sigma\in\mathscr{P}$,
\[
\lim\limits_{\epsilon\to 0}\sup_{n}H^{\sigma}_{n,p,\epsilon}=0.
\]
\end{lemma}
\begin{proof} Note that 
\[
R_{\epsilon}=\cup^{m}_{\ell=1}\left(R_{\epsilon,\ell,1}\cup R_{\epsilon,\ell,2}\right),
\]
where $R_{\epsilon,\ell,1}=R_{\epsilon}\cap \big\{|u_{2\ell}-u_{2\ell-2}|<\epsilon \big\}$ and $R_{\epsilon,\ell,2}=R_{\epsilon}\cap \big\{|v_{2\ell}-v_{2\ell-2}|<\epsilon \big\}$. So it suffices to show that
\[
\lim\limits_{\epsilon\to 0}\sup_{n}\int_{\R^{4md}}\int_{R_{\epsilon,\ell,i}} \Phi_{n,2m}(\xi,y)\exp\Big\{ -\frac{1}{2}\Var\Big( \sum^{2m}_{j=1} \xi_j\cdot  \big(B^{H,1}_{u_j} - B^{H,2}_{v_{\sigma(j)}} \big)\Big)\Big\}\,  du\, dv\, d\xi\, dy=0
\]
for all $\ell=1,2,\dots,m$ and for $i=1,2$. We will consider only the case $i=2$ and the case $i=1$ could be treated in the same way.

Let 
\[
J_t(\xi)=\int_{\{0<u_{1}<u_{2}<\cdots<u_{2m}<t\}}\exp\Big\{ -\frac{1}{2}\Var\Big( \sum^{2m}_{j=1} \xi_j\cdot  B^{H}_{u_j}\Big)\Big\}\,  du
\]
and
\[
J^{\sigma,\epsilon}_t(\xi)=\int_{\{0<u_{1}<u_{2}<\cdots<u_{2m}<t,\, |v_{2\ell}-v_{2\ell-2}|<\epsilon \}}\exp\Big\{ -\frac{1}{2}\Var\Big( \sum^{2m}_{j=1} \xi_j\cdot  B^{H}_{v_{\sigma(j)}}\Big)\Big\}\,  dv.
\]
Applying Cauchy-Schwartz inequality, we obtain
\begin{align*}
&\int_{\R^{4md}}\int_{R_{\epsilon,\ell,2}} \Phi_{n,2m}(\xi,y)\exp\Big\{ -\frac{1}{2}\Var\Big( \sum^{2m}_{j=1} \xi_j\cdot  \big(B^{H,1}_{u_j} - B^{H,2}_{v_{\sigma(j)}} \big)\Big)\Big\}\,  du\, dv\, d\xi\, dy\\
&\leq 
\Big(\int_{\R^{4md}}\Phi_{n,2m}(\xi,y)(J_t(\xi))^2\, d\xi\, dy\Big)^{\frac{1}{2}}\Big(\int_{\R^{4md}}\Phi_{n,2m}(\xi,y)(J^{\sigma,\epsilon}_t(\xi))^2\, d\xi\, dy\Big)^{\frac{1}{2}}\\
&\leq (H_{n,2m})^{\frac{1}{2}}\Big(\int_{\R^{4md}}\Phi_{n,2m}(\xi,y)(J^{\sigma,\epsilon}_t(\xi))^2\, d\xi\, dy\Big)^{\frac{1}{2}}.
\end{align*}
Note that 
\begin{align*}
&\int_{\R^{4md}}\Phi_{n,2m}(\xi,y)(J^{\sigma,\epsilon}_t(\xi))^2\, d\xi\, dy\\
&=\int_{\R^{4md}}\int_{R_{\epsilon,\ell}} \Phi_{n,2m}(\xi,y)\exp\Big\{ -\frac{1}{2}\Var\Big( \sum^{2m}_{j=1} \xi_j\cdot  \big(B^{H,1}_{u_j} - B^{H,2}_{v_j} \big)\Big)\Big\}\,  du\, dv\, d\xi\, dy,
\end{align*}
where
\[
R_{\epsilon,\ell}=R_{\epsilon,\ell,1}\cap R_{\epsilon,\ell,2}.
\]
Since $H_{n,2m}$ is uniformly bounded in $n$, we only need to show that 
\[
\lim\limits_{n\to\infty}\int_{\R^{4md}}\int_{R_{\epsilon,\ell}} \Phi_{n,2m}(\xi,y)\exp\Big\{ -\frac{1}{2}\Var\Big( \sum^{2m}_{j=1} \xi_j\cdot  \big(B^{H,1}_{u_j} - B^{H,2}_{v_j} \big)\Big)\Big\}\,  du\, dv\, d\xi\, dy=0
\]
for all $\ell=1,2,\dots,m$.

Using Lemma \ref{lnd} and then making the change of variables $w_i=u_i-u_{i-1}$, $s_i=v_i-v_{i-1}$ and $\eta_{i}=\sum^{2m}_{j=i} \xi_j$ for $i=1,2,\dots,2m$ with the convention $u_0=v_0=0$ and $\eta_{2m+1}=0$, we obtain
\begin{align} \label{epsilon1}
&\int_{\R^{4md}}\int_{R_{\epsilon,\ell}} 
\Phi_{n,2m}(\xi,y)\exp\Big\{ -\frac{1}{2}\Var\Big( \sum^{2m}_{j=1} \xi_j\cdot  \big(B^{H,1}_{u_j} - B^{H,2}_{v_j} \big)\Big)\Big\}\,  du\, dv\, d\xi\, dy \notag\\
&\leq n^{m(2-Hd)} \int_{\R^{4md}} \int_{\widehat{R}_{\epsilon,\ell}}\prod_{i=1}^{2m}|f(y_i)| \prod^{2m}_{i=1}\big|\exp\big(\iota \frac{y_i}{n^H}\cdot (\eta_{i+1}-\eta_i)\big)-1\big| \notag\\
         &\quad\quad\quad \times  \exp\Big\{-\frac{\kappa_H}{2}\sum^{2m}_{i=1} |\eta_i|^2(w_i^{2H}+s_i^{2H})\Big\}\, dw\, ds\, d\eta\, dy,       
\end{align}
where 
\[
\widehat{R}_{\epsilon,\ell}=[0,t]^{4m}\cap\Big\{\sum^{2m}_{i=1} w_i<t,\; \sum^{2m}_{i=1} s_i<t,\; w_{2\ell}+w_{2\ell-1}<\epsilon,\; s_{2\ell}+s_{2\ell-1}<\epsilon\Big\}.
\]
Using the same argument as in the proof of Proposition \ref{tight}, we can prove that the right hand side of the inequality \eref{epsilon1} is less than a constant multiple of $\epsilon^{2-Hd}$. Letting $\epsilon\to0$ completes the proof.
\end{proof}

\medskip

For $1\leq k\leq m$, we define
\begin{align*}
     O_{2m,k}&=\Big\{u,v\in [0,t]^{2m}: u_1<u_2<\cdots<u_{2m},\, v_1<v_2<\cdots<v_{2m},  \\
    &\qquad  \qquad \frac{u_{2k}-u_{2k-2}}{2}<u_{2k}-u_{2k-1}\; \text{or}\; \frac{v_{2k}-v_{2k-2}}{2}<v_{2k}-v_{2k-1}\Big\}.
\end{align*}
Recall the definition of $\Phi_{n,2m}(\xi,y)$ in \eref{Phi}. The following result states that the integral over the domain $O_{2m,k}$ does not contribute to the limit of the $2m$-th moment, which will play a fundamental role in computing the limits of even moments. 
\begin{lemma} \label{half}
For any $1\leq k\leq m$,
\begin{equation*}
\lim_{n\to\infty}\int_{\R^{4md}} \int_{O_{2m,k}}  \Phi_{n,2m}(\xi,y)\exp \Big\{ -\frac{1}{2}\Var\big( \sum^{2m}_{j=1}\xi_j\cdot (B^{H,1}_{u_j}-B^{H,2}_{v_j})\big)\Big\}\, d\xi \, du\, dv\, dy
=0.
\end{equation*}
\end{lemma}
\begin{proof} Define
\begin{align*}
     \widehat{O}_{2m,k}&=\big\{u,v\in [0,t]^{2m}: u_1<u_2<\cdots<u_{2m}, v_1<v_2<\cdots<v_{2m},\\
             & \qquad\quad  \frac{u_{2k}-u_{2k-2}}{2}<u_{2k}-u_{2k-1}, \frac{v_{2k}-v_{2k-2}}{2}<v_{2k}-v_{2k-1}\big\}.
\end{align*}

Using Cauchy-Schwartz inequality, we obtain 
\begin{align*}
&\int_{\R^{4md}} \int_{O_{2m,k}}  \Phi_{n,2m}(\xi,y)\exp \Big\{-\frac{1}{2}\Var\Big( \sum^{2m}_{j=1}\xi_j\cdot (B^{H,1}_{u_j}-B^{H,2}_{v_j})\Big)\Big\}\, d\xi \, du\, dv\, dy\\
&\leq (H_{n,2m})^{\frac{1}{2}} \bigg(\int_{\R^{4md}}  \Phi_{n,2m}(\xi,y)\int_{\widehat{O}_{2m,k}}  \exp \Big\{ -\frac{1}{2}\Var\big( \sum^{2m}_{j=1}\xi_j\cdot (B^{H,1}_{u_j}-B^{H,2}_{v_j})\big)\Big\}\,  du\, dv\, d\xi \, dy\bigg)^{\frac{1}{2}}\\
&\leq c_1\bigg(\int_{\R^{4md}}  \Phi_{n,2m}(\xi,y)\int_{\widehat{O}_{2m,k}}  \exp \Big\{ -\frac{1}{2}\Var\big( \sum^{2m}_{j=1}\xi_j\cdot (B^{H,1}_{u_j}-B^{H,2}_{v_j})\big)\Big\}\,  du\, dv\, d\xi \, dy\bigg)^{\frac{1}{2}}.
\end{align*}
So it suffices to show
\[
\lim_{n\to\infty} \int_{\R^{4md}}  \Phi_{n,2m}(\xi,y)\int_{\widehat{O}_{2m,k}}  \exp \Big\{ -\frac{1}{2}\Var\big( \sum^{2m}_{j=1}\xi_j\cdot (B^{H,1}_{u_j}-B^{H,2}_{v_j})\big)\Big\}\,  du\, dv\, d\eta \, dy=0.
\]
 
For $j=1,2,\dots, 2m$, we make the change of variables $w_j=u_j-u_{j-1}$ and $s_j=v_j-v_{j-1}$ with the convention $u_0=v_0=0$.  For $k=1,2,\dots,m$, define
\[
D_{2m,k}=\Big\{w,s\in [0,t]^{2m}: \sum^{2m}_{j=1}w_j< t,\; \sum^{2m}_{j=1}s_j < t, \; w_{2k-1}<w_{2k},\; s_{2k-1}<s_{2k}\Big\}.
\]
Using the second inequality in \eref{eq1},
\begin{align*}
&\int_{\R^{4md}} \int_{\widehat{O}_{2m,k}}  \Phi_{n,2m}(\xi,y)\exp \Big\{ -\frac{1}{2}\Var\big( \sum^{2m}_{j=1}\xi_j\cdot B^{H}_{u_j}\big)-\frac{1}{2}\Var\big( \sum^{2m}_{j=1}\xi_j\cdot B^{H}_{v_j}\big)\Big\} \, du\, dv\, d\xi \, dy  \\
&  \leq  c_2\, n^{m(2-Hd)}
\int_{\R^{2md}}  \int_{D_{2m,k}} \prod^{2m}_{j=1}
           |f(y_j)| \prod^{2m}_{j=1}\big(w_j^{2H}+s_j^{2H}\big)^{-\frac{d}{2}}   \\
&  \quad  \times  \E \bigg( \prod^{2m}_{j=1}
\Big|\exp\Big(\iota \frac{y_j\cdot X_{j+1}}{n^H\sqrt{
w_{j+1}^{2H}+s_{j+1}^{2H}}} - \iota \frac{y_j \cdot X_j}{n^H
\sqrt{w_j^{2H}+s_j^{2H}}}\Big)-1\Big| \bigg)\, dw\, ds\, dy, 
\end{align*}
where $\sqrt{\kappa_H}X_j$ ($1\leq j\leq 2m$) are independent copies of the $d$-dimensional standard normal random vector and $X_{2m+1}=0$. The rest of proof is similar to that of Proposition 3.3 in \cite{hnx}.
\end{proof}

\medskip

Recall the definition of $D$ in \eref{domain}. For $\ell=1,2,\dots, m$ and $K>0$, define 
\begin{equation} \label{dnkl}
D^n_{K,\ell}=D\cap \big\{u_{2\ell}-u_{2\ell-1}\geq K/n\;\; \text{or}\;\; v_{2\ell}-v_{2\ell-1}\geq K/n\big\}.
\end{equation}
The following result implies that the domain $D^n_{K,\ell}$ does not contribute to the limit of even moments.
\begin{lemma} \label{K} For any $\si\in\mathscr{P}$ and $\ell=1,2,\dots, m$, 
\[
\lim\limits_{K\to\infty}\limsup_{n\to\infty} \int_{\R^{4md}}
           \int_{D^n_{K,\ell}} \Phi_{n,2m}(\xi,y) \exp\Big\{-\frac{1}{2}\Var\Big(\prod_{j=1}^{2m} \xi_j\cdot(B^{H,1}_{u_j}-B^{H,2}_{v_{\sigma(j)}})\Big)\Big\}\, du\, dv\, d\xi dy=0.
\]
\end{lemma}
\begin{proof} Let $t=\max(t_1,t_2)$. Define 
\begin{align*}
\widehat{D}^n_{K,\ell}
&=\big\{u,v\in [0,t]^{2m}: u_{1}<u_{2}<\cdots<u_{2m},\, v_{1}<v_{2}<\cdots<v_{2m},\\
&\qquad\qquad  u_{2\ell}-u_{2\ell-1}\geq K/n,\; v_{2\ell}-v_{2\ell-1}\geq K/n\big\}
\end{align*}
and
\[
I^n_{K,\ell}=\int_{\R^{4md}}
           \int_{D^n_{K,\ell}} \Phi_{n,2m}(\xi,y) \exp\Big\{-\frac{1}{2}\Var\Big(\prod_{j=1}^{2m} \xi_j\cdot(B^{H,1}_{u_j}-B^{H,2}_{v_{\sigma(j)}})\Big)\Big\}\, du\, dv\, d\xi\, dy.
\]
Applying Cauchy-Schwartz ineuqality,
\[
I^n_{K,\ell}
\leq c_1\bigg(\int_{\R^{4md}}
           \int_{\widehat{D}^n_{K,\ell}} \Phi_{n,2m}(\xi,y)\exp\Big\{-\frac{1}{2}\Var\Big(\sum^{2m}_{j=1}\xi_j\cdot \big(B^{H,1}_{u_j}-B^{H,2}_{v_j}\big)\Big)\Big\}\, du\, dv\, d\xi\, dy\bigg)^{\frac{1}{2}}.                    
\]
According to Lemma \ref{half}, we can replace $\widehat{D}^n_{K,\ell}$ in the above inequality with 
\[
\widetilde{D}^n_{K,\ell}
=\widehat{D}^n_{K,\ell}\cap \Big\{ \frac{u_{2\ell}-u_{2\ell-2}}{2}>u_{2\ell}-u_{2\ell-1}\geq K/n,\;  \frac{v_{2\ell}-v_{2\ell-2}}{2}> v_{2\ell}-v_{2\ell-1}\geq K/n\Big\}.
\]
The rest of the proof is similar to that of Proposition 3.4 in \cite{hnx}.
\end{proof}

\medskip

We next divide $\mathscr{P}$ into two subsets. In section 3.4, we will show that permutations in one subset do not contribute to the convergence of even moments.

For each $\sigma$ in $\mathscr{P}$, we introduce the following decomposition of $I=\{1,2,\dots,2m\}$:
\[
\begin{array}{ll}
I^{\sigma}_{ee}=\{j\in I:\;  j\;  \text{is even and}\;  \sigma(j)\; \text{is even}\}, &
I^{\sigma}_{eo}=\{j\in I:\;  j\;  \text{is even and}\;  \sigma(j)\;  \text{is odd} \},\\
I^{\sigma}_{oe}=\{j\in I:\;  j\;  \text{is odd and}\;  \sigma(j)\;  \text{is even} \}, &
I^{\sigma}_{oo}=\{j\in I:\;  j\;  \text{is odd and}\;  \sigma(j)\;  \text{is odd} \}.
\end{array}
\]
Let $I_e=\{2j: 1\leq j\leq m\}$ and $I_o=\{2j-1:  1\leq j\leq m\}$. Then $I_e=I^{\sigma}_{ee}+I^{\sigma}_{eo}$ and $I_o=I^{\sigma}_{oe}+I^{\sigma}_{oo}$ for all $\sigma$ in $\mathscr{P}$. We make the change of variables $w_{2k}= u_{2k}$, $w_{2k-1}=n( u_{2k}-u_{2k-1})$, $s_{2k}= v_{2k}$, $s_{2k-1}=n(  v_{2k}-v_{2k-1})$ for $k=1,2,\dots,m$. Define 
\begin{align} \label{Dn}
D_n&=\Big\{ w,s\in\R^{2m}: 0<w_2<w_4<\cdots<w_{2m}<t_1, 0<s_2<s_4<\cdots<s_{2m}<t_2 \notag\\
&\qquad\qquad  0<w_{2k-1}<n(w_{2k}-w_{2k-1}),\; 0<s_{2k-1}<n(s_{2k}-s_{2k-1}),\; 1\leq k\leq m\Big\}.
\end{align}
From the above decomposition of $I$,
\begin{align} \label{dcmp}
& n^{2m} \E \Big[\int_{D} \prod_{j=1}^{2m}f\big(n^{H}B^{H,1}_{u_j}-n^{H}B^{H,2}_{v_{\sigma(j)}}\big)\, du\, dv \Big] \notag \\
=&\E \Big\{ \int_{D_n}  \prod_{j\in I^{\sigma}_{ee}}
             f\big(n^H [B^{H,1}_{w_j}- B^{H,2}_{s_{\sigma(j)}}]\big) \prod_{j\in I^{\sigma}_{eo}}f\big(n^H[B^{H,1}_{w_j}- B^{H,2}_{s_{\sigma(j)+1}}]-n^H[B^{H,2}_{s_{\sigma(j)+1}-\frac{s_{\sigma(j)}}{n}}-B^{H,2}_{s_{\sigma(j)+1}}]\big)\notag\\
     &\quad\times \prod_{j\in I^{\sigma}_{oe}}  f\big(n^H[B^{H,1}_{w_{j+1}}- B^{H,2}_{s_{\sigma(j)}}]+n^H[B^{H,1}_{w_{j+1}-\frac{w_j}{n}}-B^{H,1}_{w_{j+1}}]\big)\notag \\
     &\qquad\times  \prod_{j\in I^{\sigma}_{oo}}f\Big(n^H [B^{H,1}_{w_{j+1}}-B^{H,2}_{s_{\sigma(j)+1}}] +n^H[B^{H,1}_{w_{j+1}-\frac{w_j}{n}}-B^{H,1}_{w_{j+1}}]-n^H[B^{H,2}_{s_{\sigma(j)+1}-\frac{s_{\sigma(j)}}{n}}-B^{H,2}_{s_{\sigma(j)+1}}]\Big)\,dw \, ds\Big\}.
 \end{align}

Assume that $x_2, x_4,\dots, x_{2m}$ and $z_2, z_4,\dots, z_{2m}$ are linearly independent elements in some linear space. For any $\sigma$ in $\mathscr{P}$, let
\[
A^{\sigma}_{ee}=\big\{x_j-z_{\sigma(j)}:\, j\in I^{\sigma}_{ee}\big\},\quad A^{\sigma}_{oe}=\big\{x_{j+1} - z_{\sigma(j)}:\,  j \in I^{\sigma}_{oe}\big\};
\]
\[
A^{\sigma}_{eo}=\big\{x_j-z_{\sigma(j)+1}:\, j\in I^{\sigma}_{eo}\big\}, \quad A^{\sigma}_{oo}=\big\{x_{j+1} - z_{\sigma(j)+1}:\,  j \in I^{\sigma}_{oo}\big\}.
\]
Note that elements in each of the above sets are linearly independent. For simplicity, we use $\# A$ to denote the cardinality of a set $A$. Suppose $\#I^{\sigma}_{ee}=r$. Then $\#A^{\sigma}_{ee}=\#A^{\sigma}_{oo}=r$ and $\#A^{\sigma}_{eo}=\#A^{\sigma}_{oe}=m-r$. We are interested in the dimension of the set $A_{\sigma}:=A^{\sigma}_{ee}\cup  A^{\sigma}_{oe}\cup A^{\sigma}_{eo}\cup A^{\sigma}_{oo}$, that is, the maximum number of elements in $A_{\sigma}$ which are linearly independent. Since elements in $A^{\sigma}_{ee}\cup A^{\sigma}_{oe}$ are linearly independent, the dimension of $A_{\sigma}$ is greater than or equal to $m$. 

\begin{lemma} \label{permutation} The dimension of $A_{\sigma}$ is $m$ if and only if $\{(j,\sigma(j)):\, j\in I^{\sigma}_{ee}\}=\{(j+1, \sigma(j)+1):\,  j \in I^{\sigma}_{oo}\}$ and $\{(j+1,\sigma(j)):\,  j \in I^{\sigma}_{oe}\}=\{(j, \sigma(j)+1):\, j\in I^{\sigma}_{eo}\}$.
\end{lemma}
\begin{proof}  It suffices to show the only if part. Note that the $m$ elements in $A^{\sigma}_{ee}\cup A^{\sigma}_{oe}$ are linearly independent. If one of the two condition fails, then there must exist an element in $A^{\sigma}_{eo} \cup A^{\sigma}_{oo}$ such that it does not belong to the space spanned by $A^{\sigma}_{ee}\cup A^{\sigma}_{oe}$. This implies that the dimension of $A_{\sigma}$ is greater than $m$. 
\end{proof}

Let $\mathscr{P}_0=\{\sigma\in\mathscr{P}:\; \text{the dimension of}\; A_{\sigma}\; \text{is}\; m\}$ and $\mathscr{P}_1=\mathscr{P}-\mathscr{P}_0$. Lemma \ref{permutation} implies 
\[
\#\mathscr{P}_0=\sum^m_{r=0}\binom {m} {r} m! =m!\, 2^m.
\]

\subsection{Convergence of even moments} 
We will show the convergence of all even moments. Recall that 
\[
 \E\big[F_{n}(t_1,t_2)\big]^{2m} =(2m)!\, n^{m(2+Hd)}  \sum_{\sigma\in\mathscr{P}}\E \Big[
           \int_{D} \prod_{j=1}^{2m}f\big(n^{H}B^{H,1}_{u_j}-n^{H}B^{H,2}_{v_{\sigma(j)}}\big)\, du\, dv \Big],
\]
where 
\[
D=\Big\{0=u_0<u_{1}<u_{2}<\cdots<u_{2m}<t_1,\, 0=v_0<v_{1}<v_{2}<\cdots<v_{2m}<t_2\Big\}.
\]
Note that we can find a sequence of functions $f_N$,  which are
infinitely differentiable with compact support, such that
$\int_{\R^d} f_N(x)\, dx=0$ and
\[
      \lim_{N\to \infty} \int_{\R^d}  |f(x)-f_N(x)| \left(|x|^{\frac 2H -d} \vee1 \right)  dx=0.
\]
Therefore, by Proposition \ref{tight},  we can assume that
$f$ is infinitely differentiable with compact support and
$\int_{\R^d} f (x)\, dx=0$. 

We first show that permutations in $\mathcal{P}_1$ do not contribute to the limit of even moments using Lemmas \ref{epsilon} and \ref{K}.
\begin{proposition} \label{p1} For any $\si\in\mathscr{P}_1$,
\begin{equation} \label{permut}
\lim\limits_{n\to\infty} 
n^{m(2+Hd)}\E \Big[ \int_{D} \prod_{j=1}^{2m}f\big(n^{H}B^{H,1}_{u_j}-n^{H}B^{H,2}_{v_{\sigma(j)}}\big)\, du\, dv \Big]=0.
\end{equation}
\end{proposition}
\begin{proof}  For any $\epsilon>0$ and $K>0$, define
\[
D^n_{K,\epsilon}=D\cap\big\{ u_{2\ell}-u_{2\ell-1}<K/n,\; v_{2\ell}-v_{2\ell-1}<K/n,\; u_{2\ell}-u_{2\ell-2}\geq\epsilon,\; v_{2\ell}-v_{2\ell-2}\geq\epsilon,\; \ell=1,2,\dots, m\big\}.
\]
Thanks to Lemmas \ref{epsilon} and \ref{K}, we can replace $D$ in \eref{permut} with $D^n_{K,\epsilon}$. 

Recall the equality in \eref{dcmp}. Making proper change of variables gives
\begin{align*}
& n^{m(2+Hd)} \E \Big[\int_{D^n_{K,\epsilon}} \prod_{j=1}^{2m}f\big(n^{H}B^{H,1}_{u_j}-n^{H}B^{H,2}_{v_{\sigma(j)}}\big)\, du\, dv \Big] \notag \\
=&\E \Big\{ \int_{\widehat{D}^n_{K,\epsilon}}  \prod_{j\in I^{\sigma}_{ee}}
             f\big(n^H [B^{H,1}_{w_j}- B^{H,2}_{s_{\sigma(j)}}]\big) \prod_{j\in I^{\sigma}_{eo}}f\big(n^H[B^{H,1}_{w_j}- B^{H,2}_{s_{\sigma(j)+1}}]-n^H[B^{H,2}_{s_{\sigma(j)+1}-\frac{s_{\sigma(j)}}{n}}-B^{H,2}_{s_{\sigma(j)+1}}]\big)\notag\\
     &\quad\times \prod_{j\in I^{\sigma}_{oe}}  f\big(n^H[B^{H,1}_{w_{j+1}}- B^{H,2}_{s_{\sigma(j)}}]+n^H[B^{H,1}_{w_{j+1}-\frac{w_j}{n}}-B^{H,1}_{w_{j+1}}]\big)\notag \\
     &\qquad\times  \prod_{j\in I^{\sigma}_{oo}}f\Big(n^H [B^{H,1}_{w_{j+1}}-B^{H,2}_{s_{\sigma(j)+1}}] +n^H[B^{H,1}_{w_{j+1}-\frac{w_j}{n}}-B^{H,1}_{w_{j+1}}]-n^H[B^{H,2}_{s_{\sigma(j)+1}-\frac{s_{\sigma(j)}}{n}}-B^{H,2}_{s_{\sigma(j)+1}}]\Big)\,dw \, ds\Big\},
 \end{align*}
where
\begin{align}
\widehat{D}^n_{K,\epsilon} \label{Dnk}
&=\Big\{w,s\in \R^{2m}: \epsilon<w_{2}<w_4<\dots<w_{2m}<t_1,\; \epsilon<s_{2}<s_4<\dots<s_{2m}<t_2, \notag\\
&\qquad\qquad\qquad 0<w_{2k-1}<K\wedge n(w_{2k}-w_{2k-2}),\; 0<s_{2k-1}<K\wedge n(s_{2k}-s_{2k-2}),\; \notag\\
&\qquad\qquad\qquad\qquad\qquad\qquad\qquad w_{2k}-w_{2k-2}\geq \epsilon,\; s_{2k}-s_{2k-2}\geq \epsilon,\; k=2,\dots,m\Big\}.
\end{align}

Since $\si\in\mathcal{P}_1$, there exists a $j_0\in I^{\sigma}_{eo}\cup I^{\sigma}_{oo}$ such that $j_0\notin I^{\sigma}_{ee}\cup I^{\sigma}_{oe}$. Without loss of generality, we can assume that $j_0\in I^{\sigma}_{eo}$. Let $p_n(x,y)$ be the density function of the random field $Z_n=(X_j, Y^n_j)$ where
\begin{equation} \label{X}
X_j
=\begin{cases} 
B^{H,1}_{w_{j_0}}- B^{H,2}_{s_{\sigma(j_0)+1}} & \hbox{if $j=j_0$}\\
B^{H,1}_{w_j}- B^{H,2}_{s_{\sigma(j)}} & \hbox{if $j\in I^{\sigma}_{ee}$}\\
B^{H,1}_{w_{j+1}}- B^{H,2}_{s_{\sigma(j)}} & \hbox{if $j\in I^{\sigma}_{oe}$}
\end{cases} 
\end{equation}
and
\[
Y^n_j
=\begin{cases} 
n^H[B^{H,2}_{s_{\sigma(j_0)+1}-\frac{s_{\sigma(j_0)}}{n}}-B^{H,2}_{s_{\sigma(j_0)+1}}]
& \hbox{if $j=j_0$}\\
n^H[B^{H,1}_{w_{j+1}-\frac{w_j}{n}}-B^{H,1}_{w_{j+1}}] & \hbox{if $j\in I^{\sigma}_{oe}$}.
\end{cases} 
\]

Since $f$ is infinitely differentiable and had compact support,
\begin{align*}
&n^{m(2+Hd)}\Big|\E \Big[
           \int_{D^n_{K,\epsilon}} \prod_{j=1}^{2m}f\big(n^{H}B^{H,1}_{u_j}-n^{H}B^{H,2}_{v_{\sigma(j)}}\big)\, du\, dv \Big]\Big|\\
&\leq c_1\, n^{mHd}  \E \Big[
            \int_{\widehat{D}^n_{K,\epsilon}} \Big| f\Big(n^H[B^{H,1}_{w_{j_0}}- B^{H,2}_{s_{\sigma(j_0)+1}}]-n^H[B^{H,2}_{s_{\sigma(j_0)+1}-\frac{s_{\sigma(j_0)}}{n}}-B^{H,2}_{s_{\sigma(j_0)+1}}]\Big)\Big|\\
     &  \quad\times             \prod_{j\in I^{\sigma}_{ee}} \Big| f\Big(n^H [B^{H,1}_{w_j}- B^{H,2}_{s_{\sigma(j)}}]\Big)\Big|
 \prod_{j\in I^{\sigma}_{oe}} \Big| f\Big(n^H[B^{H,1}_{w_{j+1}}- B^2_{s_{\sigma(j)}}]+n^H[B^{H,1}_{w_{j+1}-\frac{w_j}{n}}-B^{H,1}_{w_{j+1}}]\Big)\Big| \,dw \, ds  \Big]\\
&= c_2\, n^{mHd}\int_{\widehat{D}^n_{K,\epsilon}}\int_{\R^{(m+2+\#I^{\sigma}_{oe})d}} 
\big| f(n^H x_{j_0}+y_{j_0})\big|\prod_{j\in I^{\sigma}_{ee}} \big| f(n^Hx_j)\big| \prod_{j\in I^{\sigma}_{oe}} \big| f(n^Hx_j+y_j)\big|\, p_n(x,y)\, dx\, dy \,dw \, ds\\
&=c_2\, n^{-Hd}\int_{\widehat{D}^n_{K,\epsilon}} \int_{\R^{(m+2+\#I^{\sigma}_{oe})d}} 
\big| f(x_{j_0}+y_{j_0})\big|\prod_{j\in I^{\sigma}_{ee}} \big| f(x_j)\big| \prod_{j\in I^{\sigma}_{oe}} \big| f(x_j+y_j)\big|\, p_n(\frac{x}{n^H},y)\, dx\, dy \,dw \, ds\\
&\leq c_3\, n^{-Hd} \int_{\widehat{D}^n_{K,\epsilon}} \int_{\R^{(1+\#I^{\sigma}_{oe})d}}\sup_{x}p_n(\frac{x}{n^H},y)\,dy\,dw \, ds.                   
\end{align*}

Let $Q_n(w,s)$ be the covariance matrix function of $Z_n=(X_j, Y^n_j)$ defined above. Then $Q_n$ has the following expression 
\[ 
Q_n=\left[ \begin{array}{cc}
A & C^T_n\\
C_n &  B_n
\end{array} \right],
\] 
where $A=A(w,s)$ is the covariance matrix function of $X=(X_j)$, $C_n=C_n(w,s)$ the covariance matrix function of $X=(X_j)$ and $Y=(Y_j)$, and $B_n=B_n(w,s)$ the covariance matrix function of $Y=(Y_j)$.

After doing some algebra, we have
\[ 
Q^{-1}_n=\left[ \begin{array}{cc}
(A-C^T_n B^{-1}_n C_n)^{-1} & -A^{-1}C^T_n(B_n-C_nA^{-1}C^T_n)^{-1}\\
-B_n^{-1}C_n(A-C^T_nB^{-1}_nC_n)^{-1} &  (B_n-C_nA^{-1}C^T_n)^{-1}
\end{array} \right],
\] 
and $\det(Q_n)=\det(A)\det(B_n-C_nA^{-1}C^T_n)$. For simplicity of notation, we write
\[ 
Q^{-1}_n=\left[ \begin{array}{cc}
D_1 & D^T_2\\
D_2 & D_4
\end{array} \right].
\] 
Note that 
\begin{align*}
(x,y)Q_n^{-1}(x,y)^T
&=xD_1x^T+xD^T_2y^T+yD_2x^T+yD_4y^T\\
&= xD_1x^T+2xD^T_2y^T+yD_4y^T\\
&=( x\sqrt{D_1})( x\sqrt{D_1})^T+2( x\sqrt{D_1})(\sqrt{D_1})^{-1}D^T_2y^T+yD_4y^T\\
&\geq y(D_4-D_2D^{-1}_1D^T_2)y^T.
\end{align*}
Then
\begin{align*}
\int_{\R^{(1+\#I^{\sigma}_{oe})d}}\sup_{x}p_n(\frac{x}{n^H},y)\,dy
&=\int_{\R^{(1+\#I^{\sigma}_{oe})d}}\sup_{x}p_n(x,y)\,dy\\
&=c_4\int_{\R^{(1+\#I^{\sigma}_{oe})d}}(\det(Q_n))^{-\frac{1}{2}}\sup_{x}\exp\Big\{-\frac{1}{2}(x,y)Q_n^{-1}(x,y)^T\Big\}\,dy\\
&\leq c_4 \int_{\R^{(1+\#I^{\sigma}_{oe})d}} (\det(Q_n))^{-\frac{1}{2}}\exp\Big\{-\frac{1}{2}y\big(D_4-D_2D^{-1}_1D^T_2\big)y^T\Big\}\, dy\\
&=c_5(\det(Q_n))^{-\frac{1}{2}} \big(\det(D_4-D_2D^{-1}_1D^T_2)\big)^{\frac{1}{2}}\\
&=c_5(\det(A))^{-\frac{1}{2}}.
\end{align*}
Therefore, 
\[
n^{m(2+Hd)}\Big|\E \Big[
           \int_{D^n_{K,\epsilon}} \prod_{j=1}^{2m}f\big(n^{H}B^{H,1}_{u_j}-n^{H}B^{H,2}_{v_{\sigma(j)}}\big)\, du\, dv \Big]\Big|
\leq c_6\, n^{-Hd} \int_{\widehat{D}^n_{K,\epsilon}} (\det(A(w,s)))^{-\frac{1}{2}}\,dw \, ds.                   
\]

From the definition of $X$ in \eref{X}, we see that the components of $X$ are linearly independent and thus $A(w,s)$ is not singular. Taking into account the definition of $\widehat{D}^n_{K,\epsilon}$ in \eref{Dnk} and the continuity of $\det(A(w,s))$, we obtain $\det(A(w,s))\geq c_{\epsilon}>0$ for all $(w,s)$ in $\widehat{D}^n_{K,\epsilon}$. Therefore,
\[
n^{m(2+Hd)}\Big|\E \Big[
           \int_{D^n_{K,\epsilon}} \prod_{j=1}^{2m}f\big(n^{H}B^{H,1}_{u_j}-n^{H}B^{H,2}_{v_{\sigma(j)}}\big)\, du\, dv \Big]\Big|
\leq c_7\, c^{-\frac{1}{2}}_{\epsilon}n^{-Hd}.                   
\]
This completes the proof. \end{proof}

\medskip

We next show the convergence of all even moments.
\begin{proposition} \label{even} For any $m\in\N$,
\[
\lim\limits_{n\to\infty} \E\big[F_n(t_1,t_2)\big]^{2m}=D^m_{H,d}\, \|f\|^m_{\frac{2}{H}-d}\, \E\big[\sqrt{\alpha(0,t_1,t_2)}\, \zeta  \big]^{2m}.
\]
\end{proposition} 
\begin{proof} By Proposition \ref{p1} and equation \eref{moment}, we only need to show 
\begin{align} \label{moment1}
&\lim\limits_{n\to\infty} (2m)!\, n^{m(2+Hd)}  \sum_{\sigma\in\mathscr{P}_0}\E \Big[
           \int_{D} \prod_{j=1}^{2m}f\big(n^{H}B^{H,1}_{u_j}-n^{H}B^{H,2}_{v_{\sigma(j)}}\big)\, du\, dv \Big] \notag\\
          &\qquad\qquad\qquad\qquad\qquad\qquad\qquad =D^m_{H,d}\, \|f\|^m_{\frac{2}{H}-d}\, \E\big[\sqrt{\alpha(0,t_1,t_2)}\, \zeta \big]^{2m}. 
\end{align}
The proof will be done in several steps.

\medskip \noindent
\textbf{Step 1} \quad Since $\sigma\in\mathscr{P}_0$, by Lemma \ref{permutation}, the equation \eref{dcmp} can be written as
\begin{align*}
     & n^{2m} \E \Big[
           \int_{D} \prod_{j=1}^{2m}f\big(n^{H}B^{H,1}_{u_j}-n^{H}B^{H,2}_{v_{\sigma(j)}}\big)\, du\, dv \Big] \\
    &= \E \Big[
            \int_{D_n}  \prod_{j\in I^{\sigma}_{ee}}
             f\big(n^H [B^{H,1}_{w_j}- B^{H,2}_{s_{\sigma(j)}}]\big) \prod_{j\in I^{\sigma}_{eo}}f\big(n^H[B^{H,1}_{w_j}- B^{H,2}_{s_{\sigma(j)+1}}]-n^H[B^{H,2}_{s_{\sigma(j)+1}-\frac{s_{\sigma(j)}}{n}}-B^{H,2}_{s_{\sigma(j)+1}}]\big)\\
     &  \quad\times \prod_{j\in I^{\sigma}_{eo}}  f\big(n^H[B^{H,1}_{w_{j}}- B^{H,2}_{s_{\sigma(j)+1}}]+n^H[B^{H,1}_{w_{j}-\frac{w_{j-1}}{n}}-B^{H,1}_{w_{j}}]\big)\\
     &\qquad\times  \prod_{j\in I^{\sigma}_{ee}}f\big(n^H [B^{H,1}_{w_{j}}-B^{H,2}_{s_{\sigma(j)}}] +n^H[B^{H,1}_{w_{j}-\frac{w_{j-1}}{n}}-B^{H,1}_{w_{j}}]-n^H[B^{H,2}_{s_{\sigma(j)}-\frac{s_{\sigma(j)-1}}{n}}-B^{H,2}_{s_{\sigma(j)}}]\big)\,dw \, ds  \Big].
\end{align*}

We introduce random fields $X^n(w,s)=\big\{X^n_{j}(w,s): j\in I_{e}\big\}$ and $Y^n(w,s)=\big\{Y^n_j(w,s): j\in I_{e}\big\}$ with
\[
X^n_j(w,s)=
\begin{cases} 
B^{H,1}_{w_j}- B^{H,2}_{s_{\sigma(j)}}, & \quad \hbox{if $j\in I^{\sigma}_{ee}$},\\
\\
[B^{H,1}_{w_j}- B^{H,2}_{s_{\sigma(j)+1}}]-[B^{H,2}_{s_{\sigma(j)+1}-\frac{s_{\sigma(j)}}{n}}-B^{H,2}_{s_{\sigma(j)+1}}], & \quad \hbox{if $j\in I^{\sigma}_{eo}$},
\end{cases}
\]
and
\[
Y^n_j(w,s)=
\begin{cases} 
n^H[B^{H,1}_{w_{j}-\frac{w_{j-1}}{n}}-B^{H,1}_{w_{j}}]-n^H[B^{H,2}_{s_{\sigma(j)}-\frac{s_{\sigma(j)-1}}{n}}-B^{H,2}_{s_{\sigma(j)}}], & \quad \hbox{if $j\in I^{\sigma}_{ee}$},\\
\\
n^H[B^{H,1}_{w_{j}-\frac{w_{j-1}}{n}}-B^{H,1}_{w_{j}}]+n^H[B^{H,2}_{s_{\sigma(j)+1}-\frac{s_{\sigma(j)}}{n}}-B^{H,2}_{s_{\sigma(j)+1}}], &\quad \hbox{if $j\in I^{\sigma}_{eo}$}.
\end{cases}
\]

Let $Z_n(w,s)=\big(X^n(w,s), Y^n(w,s)\big)$. Denote the covariance matrix and
the probability density function of the Gaussian random field $Z_n(w,s)$  by $Q_{n}(w,s)$ and
\[
p_{n}(x,y) = (2\pi )^{-md}\big(\det
Q_{n}(w,s)\big)^{-\frac{1}{2}}\exp\Big\{ -\frac{1}{2} (x,y)Q_{n}(w,s)^{-1}(x,y)^{T}\Big\},
\]
respectively. Then
\begin{align}      \label{moment2}
& n^{m(2+Hd)} \E \Big[
           \int_D \prod_{j=1}^{2m}f\big(n^{H}B^{H,1}_{u_j}-n^{H}B^{H,2}_{v_{\sigma(j)}}\big)\, du\, dv \Big] \notag\\     
     & =n^{mHd}\E \Big[
            \int_{D_n} \prod_{j\in I_{e}}f\big(n^H X^n_j(w,s)\big)f\big(n^H X^n_j(w,s)+Y^n_{j}(w,s)\big)\,dw \, ds \big] \notag \\
     &=n^{mHd}\int_{\R^{2md}} \int_{D_n}\prod_{j\in I_{e}}f(n^H x_j)f(n^Hx_j+y_{j-1})\, p_n(x,y)\,dw \, ds\,dx \,dy\notag \\
     &= \int_{\R^{2md}} \int_{D_n}  F(x,y)\, p_n(\frac{x}{n^H},y)\,dw \, ds\, dx\, dy,  
     \end{align}
where $F(x,y)=\prod\limits_{j\in I_{e}} f(x_j )f(x_j +y_{j})$.

We need to compute the limit of the density $p_n(x/n^H,y)$ as $n$ tends to infinity.
The covariance
matrix  between the components of $X^n(w,s)$ and $Y^n(w,s)$ converges to
the zero matrix, and the covariance matrix of the random field  $Y^n(w,s)$  converges to
a diagonal matrix with entries equal to  $w_{j-1}^{2H}+s_{\sigma(j)-1}^{2H}$ when $j\in I^{\sigma}_{ee}$ and  $w_{j-1}^{2H}+s_{\sigma(j)}^{2H}$ when $j\in I^{\sigma}_{eo}$. Let $A^{\sigma}(w,s)$ be the covariance matrix of $X(w,s)=(X_j(w,s): j\in I_e)$ with
\[
X_j(w,s)=
\begin{cases} 
B^{H,1}_{w_j}- B^{H,2}_{s_{\sigma(j)}}, & \quad \hbox{if $j\in I^{\sigma}_{ee}$},\\
\\
B^{H,1}_{w_j}- B^{H,2}_{s_{\sigma(j)+1}}, & \quad \hbox{if $j\in I^{\sigma}_{eo}$}.
\end{cases}
\]
We see that the covariance matrix of the random field $X^n(w,s)$ converges to $A^{\sigma}(w,s)$. Thus,
\begin{align*}
\lim_{n\rightarrow \infty} p_{n}(\frac{x}{n^H},y) &= (2\pi )^{-
md} \big(\det A^{\sigma}(w,s)\big)^{-\frac{1}{2}}\\
&\qquad  \times
\prod_{j\in I_{ee}}  \big(w_{j-1}^{2H}+s_{\sigma(j)-1}^{2H}\big)^{-\frac{d}{2}}
\exp\Big(-\frac{1}{2}\frac{|y_{j}|^2} {w_{j-1}^{2H}+s_{\sigma(j)-1}^{2H} } \Big)\\
&\qquad\qquad \times
\prod_{j\in I_{eo}}  \big(w_{j-1}^{2H}+s_{\sigma(j)}^{2H}\big)^{-\frac{d}{2}}
\exp\Big(-\frac{1}{2}\frac{|y_{j}|^2} {w_{j-1}^{2H}+s_{\sigma(j)}^{2H} } \Big).
 \end{align*}
On the other hand, the region  $D_n$ converges, as $n$
tends to infinity, to
\[
\Big\{ w,s\in\R_+^{2m}:  0<w_2< w_4<\cdots<w_{2m} <t_1,\,   0<s_2< s_4<\cdots<s_{2m} <t_2\Big\}.
\]
Note that we can add a term $-1$ because $\int_{\R^d} F(x,y)\, dy_{j} =0$ for all $j$ in $I_e$ and
\begin{align*}
&\int^{\infty}_0\int_{0}^\infty    (w^{2H}+s^{2H})^{-\frac{d}{2}} \big( e^{ -\frac{1}{2}\frac  { |y_{j}|^2}{w^{2H}+s^{2H}}} -1 \big)\, dw\, ds\\
& =-\, |y_{j}|^{\frac 2H-d} \int_{0}^\infty\int_{0}^\infty
(w^{2H}+s^{2H})^{-\frac{d}{2}}  \big( 1- e^{ -\frac{1}{2}\frac{1} {w^{2H}+s^{2H}} }\big)\, dw\, ds.
\end{align*}
Therefore, provided that we can interchange the limit and the
integrals in the expression \eref{moment2}, we obtain that the limit equals 
\begin{equation} \label{limit}
\frac{D_{H,d}^{m}}{4^m} 
 \| f\|_{\frac 2H-d}^{m}
(2\pi )^{-\frac{md}{2}} \int_{O}  \big(\det A^{\sigma}(w,s)\big)^{-\frac{1}{2}}  dw\, ds,
\end{equation}
where 
\[
O=\Big\{0<w_2< w_4<\cdots<w_{2m} <t_1,\,   0<s_2< s_4<\cdots<s_{2m} <t_2\Big\}.
\]
Finally, the left hand side of \eref{moment1} equals
\begin{align*}  \label{equ3}
&(2m)!\, \sum_{\sigma\in\mathscr{P}_0}\frac{D_{H,d}^{m}}{4^m} 
 \| f\|_{\frac 2H-d}^{m}
(2\pi )^{-md} \int_{O}  (\det A^{\sigma}(w,s))^{-\frac{1}{2}}  dw\, ds \notag\\
&=
(2m-1)!! D_{H,d}^{m}\,
 \| f\|_{\frac 2H-d}^{m} \int_{E^m} (2\pi )^{- \frac{md}2} (\det A (u, v))^{-\frac{1}{2}}  du\, dv,
\end{align*}
and, taking into account of Lemma \ref{lema2.1},  this would finish the proof.

\medskip \noindent
\textbf{Step 2} \quad Recall the notation $D_n$ in \eref{Dn}. Define
\[
      D_{n,K}=D_n\cap \Big\{0< w_{2k-1} <K\wedge n(w_{2k}-w_{2k-2}),\, 0< s_{2k-1} <K\wedge n(s_{2k}-s_{2k-2}),\,1\le k\le m\Big\}.
\]
The region $D_{n,K}$ is uniformly
bounded in $n$ and we can then  interchange the limit and the
integral with respect to $w$ and $s$, provided that we have a uniform
integrability condition.  

Observe that
\begin{align*}
&\int_{D_{n,K}} \big| p_n(\frac{x}{n^H},y)\big|^p\, dw\, ds\\
&\leq c_1\int_{D_{n,K}}\big(\det Q_{n}(w,s)\big)^{-\frac{p}{2}}\, ds\, ds\\
&= c_2\int_{D_{n,K}} \bigg(\int_{\R^{2md}}\exp\Big\{-\frac{1}{2}\Var\Big(\sum_{j\in I_{e}}\xi_j\cdot X^n_{j}(w,s)+\sum_{j\in I_{e}}\eta_j\cdot Y^n_{j}(w,s)\Big)\Big\} d\xi\, d\eta \bigg)^{p}\, dw\, ds
\end{align*}
and
\[
\Var\Big(\sum_{j\in I_{e}}\xi_j\cdot X^n_{j}(w,s)+\sum_{j\in I_{e}}\eta_j\cdot Y^n_{j}(w,s)\Big)=I_1(\xi,\eta)+I_2(\xi,\eta),
\]
where 
\[
I_1(\xi,\eta)=\Var\Big(\sum_{j\in I_{e}}\xi_j\cdot B^{H,1}_{w_j}+\sum_{j\in I_{e}}\eta_j\cdot n^H[B^{H,1}_{w_{j}-\frac{w_{j-1}}{n}}-B^{H,1}_{w_{j}}]\Big)
\]
and
\begin{align*}
I_2(\xi,\eta)
&=\Var\Big(\sum_{j\in I^{\sigma}_{ee}}\xi_j\cdot B^{H,2}_{s_{\sigma(j)}}+\sum_{j\in I^{\sigma}_{eo}}\xi_j\cdot (B^{H,2}_{s_{\sigma(j)+1}}+[B^{H,2}_{s_{\sigma(j)+1}-\frac{s_{\sigma(j)}}{n}}-B^{H,2}_{s_{\sigma(j)+1}}])\\
&\quad\quad+\sum_{j\in I^{\sigma}_{ee}}\eta_j\cdot n^H[B^{H,2}_{s_{\sigma(j)}-\frac{s_{\sigma(j)-1}}{n}}-B^{H,2}_{s_{\sigma(j)}}]+\sum_{j\in I^{\sigma}_{eo}}\eta_j\cdot n^H[B^{H,2}_{s_{\sigma(j)+1}-\frac{s_{\sigma(j)}}{n}}-B^{H,2}_{s_{\sigma(j)+1}}]\Big).
\end{align*}

Applying Cauchy-Schwartz inequality gives
\begin{align*}
&\int_{\R^{2md}}\exp\Big\{-\frac{1}{2}\big[I_1(\xi,\eta)+I_2(\xi,\eta)\big]\Big\}\, d\xi\, d\eta\\
&\leq \Big(\int_{\R^{2md}}\exp\big\{-I_1(\xi,\eta)\big\}\, d\xi\, d\eta\Big)^{\frac{1}{2}}\Big(\int_{\R^{2md}}\exp\big\{-I_2(\xi,\eta)\big\}\, d\xi\, d\eta\Big)^{\frac{1}{2}}.
\end{align*}
Using similar arguments as in the proof of Proposition 3.4 in \cite{hnx}, we obtain
\[
\int_{\R^{2md}}\exp\big\{-I_1(\xi,\eta)\big\}\, d\xi\, d\eta
\leq c_3 \prod_{k=1}^{m} (u_{2k-1})^{-Hd} \big(u_{2k } -\frac{u_{2k-1}}n -  u_{2k-2}\big)^{-Hd}
\]
and 
\[
\int_{\R^{2md}}\exp\big\{-I_2(\xi,\eta)\big\}\, d\xi\, d\eta
\leq c_4 \prod_{k=1}^{m}  (v_{2k-1})^{-Hd} \big(v_{2k } -\frac{v_{2k-1}}n -  v_{2k-2}\big)^{-Hd}.
\]

Therefore, for all $p$ such that $1\leq p<\frac{2}{Hd}$,
\begin{align*}
\sup_n\int_{D_{n,K}} \big| p_n(\frac{x}{n^H},y)\big|^p\, dw\, ds
\leq c_5,
\end{align*}
where $c_5$ is a positive constant independent of $n$, $x$ and $y$.

Let 
\[
I_{n,K}=\int_{\R^{2md}} \int_{D_{n,K}}  F(x,y)\, p_n(\frac{x}{n^H},y)\, dw\, ds\, dx\, dy.
\]
Thus, taking into account that the function $F(x,y)$ is continuous and has compact support, by the dominated convergence theorem, we obtain
 \[
 \lim_{n\to \infty} I_{n,K}=\int_{\R^{2md}} F(x,y) \Big( \lim_{n\to \infty} \int_{D_{n,K}}   p_n(\frac{x}{n^H},y)\, dw\, ds\Big)\, dx\, dy.
 \]

On the other hand, there exists $p>1$ such that
\[
 \sup_n \int_{D_{n,K}} |p_n(\frac{x}{n^H},y)|^p\, dw\, ds<\infty,
\]
 which implies
 \[
 \lim_{n\to \infty} I_{n,K}=\int_{\R^{2md}}\int_{\R^{2m}} F(x,y) \lim_{n\to \infty} \mathbf{1}_{D_{n,K}} ((w,s))\, p_n(\frac{x}{n^H},y)\, dw\, ds\, dx\, dy.
 \]
With the same notation as in \textbf{Step 1} we get
\begin{align*} \label{cmpt}
 \lim_{n\to \infty} I_{n,K}
  &=     \frac{(2m)! }{(2\pi )^{md}}
   \Big(\int_{O}  \big[\det A^{\sigma}(w, s)\big]^{-\frac{1}{2}}  dw\, ds \Big) \notag  \\
  & \times \int_{\R^{2md}} F(x,y)  \prod\limits_{j\in I_e} \int_0^K\int^{K}_{0} (u^{2H}+v^{2H})^{-\frac{d}{2}} \big( e^{-\frac{1}{2}\frac{|y_j|^2}{u^{2H}+v^{2H}}}-1\big)\, du\, dv\, dx\, dy.
\end{align*}
The right hand side of the above equality converges to the term in \eref{limit} as $K$ tends to
infinity.

\medskip \noindent
\textbf{Step 3} \quad We need to show that
\begin{equation} \label{unbd}
\lim_{K\to\infty}\limsup_{n\to\infty}\int_{\R^{2md}} \int_{D_{n}-D_{n,K}}  F(x,y)\, p_n(\frac{x}{n^H},y)\,dw \, ds\, dx\, dy=0. 
\end{equation}
Recall the equation \eref{moment2} and the notation $D^n_{K,\ell}$ in \eref{dnkl}.
\begin{align*}
&\int_{\R^{2md}} \int_{D_{n}-D_{n,K}}  F(x,y)\, p_n(\frac{x}{n^H},y)\,dw \, ds\, dx\, dy\\
&\qquad=n^{m(2+Hd)} \E \Big[
           \int_{\cup^m_{\ell=1} D^n_{K,\ell}} \prod_{j=1}^{2m}f\big(n^{H}B^{H,1}_{u_j}-n^{H}B^{H,2}_{v_{\sigma(j)}}\big)\, du\, dv \Big].
\end{align*}
Therefore, the statement in \eref{unbd} follows from Lemma \ref{K}. The proof is completed. \end{proof}

\medskip

{\medskip \noindent \textbf{Proof of Theorem \ref{thm1}.}}  This follows
from Propositions \ref{tight}, \ref{odd1} and \ref{even} by the method of moments.

\section{Appendix}

Here we give some lemmas which are necessary in the proof of Theorem \ref{thm1}.

\begin{lemma}  \label{lema2} Assume that $1<Hd<2$. There exists a positive contant $c$ such that 
\[
\int^a_0\int^b_0  (w^{2H}+s^{2H})^{-\frac{d}{2}}\, dw\, ds \leq c\, (a\wedge b)^{2-Hd}.
\]
\end{lemma}
\begin{proof} 
Without loss of generality, we can assume that $a\leq b$. Making the change of variable $v=s/w$ gives
\begin{align*}
\int^a_0\int^b_0  (w^{2H}+s^{2H})^{-\frac{d}{2}}\, dw\, ds
&=\int^a_0\int^{\frac{b}{w}}_0 w^{1-Hd}  (1+v^{2H})^{-\frac{d}{2}}\, dv\, dw \\
& \leq \int^{a}_0\int^{\infty}_0 w^{1-Hd}  (1+v^{2H})^{-\frac{d}{2}}\, dv\, dw\\
&\leq c_1a^{2-Hd},
\end{align*}
where $c_1$ is a positive constant independent of $b$.
\end{proof}

\begin{lemma} \label{lema3}  Assume that $2-H<Hd<2$. Let $X$ be a $d$-dimensional centered normal random vector with covariance matrix $\sigma^2I$. Then, for any $n\in\N$ and $y\in\R^d$, there exists a positive constant $c$ depending only on $H$ and $d$ such that
\[
\int^{\infty}_0\int^{\infty}_{0} (w^{2H}+s^{2H})^{-\frac{d}{2}}\, \E\Big|\exp\big(\iota\frac{y\cdot X}{n^H\sqrt{w^{2H}+s^{2H}}}\big)-1\Big|\, dw\, ds \leq c\, n^{Hd-2}|y|^{\frac{2}{H}-d}.
\]
\end{lemma}
\begin{proof} It suffices to show the above inequality when $y\neq 0$. Making the change of variables $u=|y|^{-\frac{1}{H}}nw$ and $v=|y|^{-\frac{1}{H}}ns$ gives
\begin{align*}
&\int^{\infty}_0\int^{\infty}_{0} (w^{2H}+s^{2H})^{-\frac{d}{2}}\, \E\Big|\exp\big(\iota\frac{y\cdot X}{n^H\sqrt{w^{2H}+s^{2H}}}\big)-1\Big|\, dw\, ds\\
&=n^{Hd-2}|y|^{\frac{2}{H}-d} \int^{\infty}_0\int^{\infty}_{0} (u^{2H}+v^{2H})^{-\frac{d}{2}}\, \E\Big|\exp\big(\iota\frac{y\cdot X}{|y|\sqrt{u^{2H}+v^{2H}}}\big)-1\Big|\, du\, dv\\
&\leq n^{Hd-2}|y|^{\frac{2}{H}-d} \int^{\infty}_0\int^{\infty}_{0} (u^{2H}+v^{2H})^{-\frac{d}{2}}\Big(2\wedge  (u^{2H}+v^{2H})^{-\frac{1}{2}}\E|X| \Big) \, du\, dv\\
&= c\, n^{Hd-2}|y|^{\frac{2}{H}-d}.
\end{align*}
The last equality follows from using polar coordinates and the assumption $2-H<Hd<2$.
\end{proof}

\end{document}